\documentclass[a4paper,reqno,12pt]{amsart}
\usepackage{setspace} 
\usepackage{amsfonts,amssymb,latexsym,xspace,epsfig,graphics,color}
\usepackage{amsmath,enumerate,stmaryrd,xy}
\usepackage{bbm}
\usepackage{float}
\usepackage{natbib}
\usepackage{mathrsfs}
\numberwithin{figure}{section}
\usepackage{subfigure}
\usepackage{dsfont}
\usepackage{a4wide}
\usepackage{comment}
\usepackage{enumitem} 
\usepackage{tikz}
\usepackage{tkz-berge}
\usetikzlibrary{arrows,backgrounds,trees}
\usetikzlibrary{matrix,decorations.pathreplacing,positioning}
\usepackage{multicol}
\usepackage{tkz-graph}
\usetikzlibrary{decorations.markings}
\tikzstyle directed=[postaction={decorate,decoration={markings,
    mark=at position .65 with {\arrow{stealth}}}}]
\tikzstyle reverse directed=[postaction={decorate,decoration={markings,
    mark=at position .65 with {\arrowreversed{stealth};}}}]

\usepackage{hyperref}
\hypersetup{ colorlinks,
linkcolor=blue,
filecolor=green,
urlcolor=blue,
citecolor=blue }

\usepackage{stmaryrd}
\usepackage{MnSymbol}
\newtheorem{theorem}{Theorem}[section]    
\newtheorem*{tma1}{Theorem \ref{D1o3}}  
\newtheorem*{tma2}{Theorem \ref{action}} 
\newtheorem*{tma3}{Theorem \ref{orbitas}} 
\newtheorem*{cor1}{Corollary \ref{AlgD}} 

\newtheorem{prop}[theorem]{Proposition} \newtheorem{lemma}[theorem]{Lemma}
\newtheorem{corollary}[theorem]{Corollary}

\newtheorem{remark}[theorem]{Remark}
\newtheorem{exa}[theorem]{Example}
\def\aut {\mathop {\rm Aut}\nolimits}
\def\sym {\mathop {\rm Sym}\nolimits}
\def\AA {\mathcal{A}}
\def\DD {\mathcal{D}}
\def\KK {\mathbb{K}}

\makeatletter
\@namedef{subjclassname@2020}{\textup{2020} Mathematics Subject Classification}
\makeatother

\bibdata{prob}
\bibstyle{alpha} 

\title[ Group actions and Automorphisms of  evolution algebras ]{Group actions and Automorphisms of  evolution algebras associated to finite graphs \\ 
}

\author[Mary Luz Rodi\~no Montoya]{Mary Luz Rodi\~no Montoya}
\address{Mary Luz Rodi\~no Montoya: Instituto de Matem\'aticas - Universidad de Antioquia, Calle 67 N$^{\circ}$ 53-108, Medell\'in, Colombia}
\email{mary.rodino@udea.edu.co}

\author[Natalia A. Viana Bedoya]{Natalia A. Viana Bedoya}
\address{Natalia A. Viana Bedoya: Departamento de Matemática - Universidade Federal de S\~ao Carlos, Rod. Washington Luíz, s/n - Monjolinho, São Carlos - SP, Brazil.}
\email{nvbedoya@ufscar.br}

\author[Carlos Henao]{Carlos Henao}
\address{Carlos Henao: Instituto de Matem\'aticas - Universidad de Antioquia, Calle 67 N$^{\circ}$ 53-108, Medell\'in, Colombia  }
\email{carlos.henao1@udea.edu.co}


\begin{document}
\begin{abstract} Given an  evolution algebra associated to a connected  finite graph $\Gamma$, we 
exhibit  a free action of the group of symmetries of $\Gamma$ on the set of automorphisms of the algebra.
This allows us to explicitly describe this set and we
prove that a sufficient condition for it  to be finite  is that every automorphism  is induced by a graph symmetry. Consequently, we extend
a known result about perfect evolution algebras to other families.  
\end{abstract}   
\maketitle
\section*{Introduction}
  In this work  we  describe the set of automorphisms of an  evolution algebra, associated to a finite, connected,  undirected graph, as the disjoint union of copies of the group of symmetries of the graph. This allows us to recognize evolution algebras
whose automorphism group is finite with a weaker hypothesis than that of being perfect or having a unique basis.
 
 Evolution algebras    were introduced   by J. P. Tian  in \cite{tian1, tian2} 
 as an algebraic way to model the self-reproduction of alleles in non-Mendelian genetics. Since then, many works have been dedicated to understand their algebraic structures. 
Formally, a  {\it  finite-dimensional evolution algebra} $\AA$ is  a $\KK$-algebra,  $\KK$ a field,  that admits a finite basis $B:=\{e_i \}_{i=1}^n$ such that $e_i^2 = \sum_{k=1}^n c_{ik} e_k$, for $c_{ik} \in \KK$  and  $e_ie_j=0$ if $i \neq j$. A basis with this property is a {\it natural basis for $\AA$}. The scalars $c_{ik} \in \KK$ are  the {\it  structure constants of  $\AA$ relative to $B$} and the matrix $M_{B} = (c_{ik})$ is  {\it the
structure matrix   of $\AA$ relative to $B$.}

Given a  finite graph $\Gamma$, 
{\it the finite-dimensional evolution algebra associated to} $\Gamma$ is the evolution $\mathbb{K}$-algebra, denoted by   $\AA_{\Gamma}$,  such that its structure matrix  is the adjacency matrix of $\Gamma$.  There is also   the {\it evolution algebra associated  to the symmetric random walk (SRW) on} $\Gamma$, denoted by  $\AA_{RW(\Gamma)}$, that is the evolution $\mathbb{K}$-algebra with structure matrix equal to   the  transition matrix of $\Gamma$.

 Finite-dimensional evolution algebras associated to  graphs      have been extensively studied under the hypothesis of idempotency: $\AA=\AA^2$ ,  see for example   \cite{TYCC,PMP, CLTV,Eld-Lab2015, Eld-Lab2021, Sri-Zou2022}. This hypothesis defines the class of {\it perfect evolution algebras}, also called regular or idempotent.
 In this work, we will deal with the evolution algebras in 
 $\{\AA_{\Gamma}, \AA_{RW(\Gamma)}\}_{\Gamma \in \mathscr{G}}$,  where $\mathscr{G}$ is the set of finite, connected, simple,  undirected graph  without loops. In particular, this class of algebras is  not contained in the class of  perfect evolution algebras.

In  \cite{Eld-Lab2015, Eld-Lab2021}, A. Elduque and  A. Labra  show  that the diagonal subgroup of the automorphism group of a perfect evolution algebra is isomorphic to diagonal group of the associated graph.  In this work, for $\AA \in 
  \{\AA_{\Gamma},\AA_{RW(\Gamma)}\}_{\Gamma \in \mathscr{G}}$, we show (see Proposition \ref{eqaybDiag}) that  these groups are still isomorphic even if $\AA$ is  not perfect. As a consequence we obtain  for $\AA \in 
  \{\AA_{\Gamma},\AA_{RW(\Gamma)}\}_{\Gamma \in \mathscr{G}}$, if $ \aut (\AA)$ denotes its    automorphisms group, that:

  \begin{tma1}\label{teo1} If $\AA$ is an evolution $\mathbb{K}$-algebra for $\mathbb{K}$ algebraically closed and $\mathcal{D}$ is the diagonal subgroup of $ \aut (\AA)$, then
 $\DD \neq \{id\}$  if and only if  $\Gamma$  is  bipartite.
\end{tma1}

\begin{cor1}
      Under the hypothesis  of Theorem \ref{D1o3}, either $\DD=\{Id\}$  or  $\DD \cong \mathbb{Z}_3$.  
      \end{cor1}
      
   If $\sym(\Gamma)$ denotes {\it the group of symmetries of $\Gamma$}, our main result for $\AA \in 
\{\AA_{\Gamma},\AA_{RW(\Gamma)}\}_{\Gamma \in \mathscr{G}}$ is: 
  \begin{tma2}
$ \sym(\Gamma)$ acts freely on  $ \aut (\AA)$ and  the orbits  are isomorphic to $\sym(\Gamma)$.
\end{tma2}
Also in  \cite{Eld-Lab2015, Eld-Lab2021}, they  show  that the automorphism group of a perfect evolution algebra is finite. This result was revisited  with different approaches  in \cite{CLTV},  in terms of regular evolution algebras,   and in
  \cite{Sri-Zou2022},   in terms of idempotent evolution algebras. 
   In this work we extend it to other families in $\{\AA_{\Gamma}, \AA_{RW(\Gamma)}\}_{\Gamma \in \mathscr{G}}$,  via the action in Theorem \ref{action}.     Specifically, we extend the result to the family of evolution algebras whose automorphisms are all induced by
symmetries of the associated graph, see \S \ref{acts}.
This family
  strictly contains the family of  those  that are  non-degenerate and   with a unique basis, i.e.,  evolution algebras with natural basis $\{e_i\}_{i=1}^n$ satisfying the {\it 2LI property}: $\{e_i^2,e_j^2\}$ is linearly independent for $i \neq j$,  (see \cite{BYM, Eld-Lab2021, CH}). In particular, it contains the family of those that are perfect. This will be shown in Example \ref{ex:indper} and  Example \ref{GraphT41} using evolution algebras that   are  not perfect, non-degenerate or {\it 2LI}, but whose automorphisms are all induced by symmetries of the respective graphs. This further extends the scope of the following theorem. 
 
 Denoting by $O_{f}$  the {\it orbit} of $f \in \aut(\AA)$  by the action, we state:

\begin{tma3}
 Let $\DD \leq \aut(\AA)$ be the diagonal subgroup. If every element of $\aut(\AA)$  is induced by a symmetry of $\Gamma$ then $\aut(\AA)= \bigsqcup_{f\in \DD} O_f.$ Moreover,  $\aut(\AA)$ is finite and
    $|\aut(\AA)|=|\DD||\sym(\Gamma)|$.
\end{tma3}

The paper is divided into three sections.
In  \S \ref{Prem}, we set the basics of definitions, notation and terminology.
Section  \S \ref{Res} is the main section, it is devoted to show the principal results.
 In \S \ref{SRW}, we point out some applications of the results in \S \ref{Res}, to  relate the group of automorphisms of $\AA_{\Gamma}$ and $ \AA_{RW(\Gamma)}$.

\section{Preliminary definitions and notation}\label{Prem}
Let $\KK$ be  an algebraically closed field  with zero characteristic and
 $\AA$  a finite-dimensional evolution $\KK$-algebra with natural basis $B=\{e_i \}_{i=1}^n$ and structure matrix $M_B$. We will  say that $\AA$ is    {\it perfect} if and only if  ${\rm det}M_B \neq 0$ (see \cite{Eld-Lab2015}, Proposition 4.2), and that  $\AA$ is  {\it non-degenerate} if $e_i^2 \neq 0$ for all $i$. Perfect evolution algebras are non-degenerate (see \cite{TYCC}).

We will denote by $\Gamma=(V,E)$  a {\it finite undirected graph}, with vertex set $V=\{1, \dots, n\}$  and edge set $E \subseteq \{\{i,j\} | i,j \in V\}$. The {\it adjacency matrix} of the graph $\Gamma$ is a symmetrical matrix  $A=(a_{ij})$ made up of zeros and ones, with $a_{ij}=1$ if $\{i,j\} \in E$, and $a_{ij}=0$ in other case.
We say that $\Gamma=(V,E')$ is a bi-directed graph if  $E':= \{(i,j),(j,i) \, | \,  \{i,j\} \in E\}$ is the edge set. Notice that the adjacency matrix for $\Gamma$ undirected or bi-directed will be the same.
The {\it transition matrix} of the graph $\Gamma$ is a matrix  $P=(p_{ij})$ with $p_{ij}=\frac{1}{d_i}$ if $\{i,j\}\in E$, and $p_{ij}=0$ in other case, where $d_i$ is the {\it degree} of the vertex $i$, i.e., the number of edges that are incident to the vertex $i$. 

\begin{exa} \label{ex:p3}   
Let $P_3$ denote the path graph  with three vertices. The matrices
$$\begin{pmatrix}
    0 & 1 & 0 \\
    1 & 0 & 1 \\
    0 & 1 & 0
\end{pmatrix} \hspace{0.5cm} \textrm{and} \hspace{0.5cm} \begin{pmatrix}
    0 & 1 & 0 \\
    \frac{1}{2} & 0 & \frac{1}{2} \\
    0 & 1 & 0
\end{pmatrix}$$ are respectively the adjacency and the transition matrices,   when we represent $P_3$ either by an undirected graph  or by a bi-directed graph (see Fig. \ref{FIG:P3}).
\begin{figure}[!h]\label{P3}
\begin{center}
\subfigure[The undirected path graph $P_{3}$.]{\begin{tikzpicture}[scale=0.4]
\draw (3.9,0) node[above,font=\footnotesize] {$1$} -- (6.9,0);
\draw (6.9,0) node[above,font=\footnotesize] {$2$} -- (9.9,0) node[above,font=\footnotesize] {$3$};
\filldraw [black] (3.9,0) circle (3.5pt);
\filldraw [black] (6.9,0) circle (3.5pt);
\filldraw [black] (9.9,0) circle (3.5pt);
\hspace{-1.2cm}
\end{tikzpicture}
\label{FIG:P3(a)}}\qquad \qquad\subfigure[The bi-directed path graph $P_{3}$.]{ \begin{tikzpicture}[scale=0.4]
\draw [thick, directed]  (13.6,0.3) to (16.5,0.3);
\draw [thick, directed]  (16.5,-0.3) to (13.6,-0.3);
\draw [thick, directed]  (17.5,0.3) to (20.5,0.3);
\draw [thick, directed]  (20.5,-0.3) to (17.5,-0.3);
\draw (13,0) node[above,font=\footnotesize] {$1$} (17.5,0);
\draw (17,0) node[above,font=\footnotesize] {$2$} (21,0) node[above,font=\footnotesize] {$3$};
\filldraw [black] (13,0) circle (3.5pt);
\filldraw [black] (17,0) circle (3.5pt);
\filldraw [black] (21,0) circle (3.5pt);
\hspace{-1.2cm}
\end{tikzpicture}
\label{FIG:P3(b)}}
\caption{ Path graph $P_3$}\label{FIG:P3}
\end{center}
\end{figure}
\end{exa} 
From now on, the graph $\Gamma$ can denote both, an undirected graph or a bidirected graph. Also, all the graphs we consider in this work are simple, connected, and without loops. If $\pi$ is an edge path  between the vertices $i$ and $j$,  define the {\it length} of $\pi$, and denote it by $l(\pi)$, the number of edges in the path $\pi$. Define the {\it distance} between $i$ and $j$ by: 
\begin{equation}\label{distance}
d(i,j):=min\{l(\pi)| \pi \textrm{ is a path between $i$ and $j$} \}, 
\end{equation}
and set $d(i,i):=0.$
A {\it cycle} in $\Gamma$ is a path that starts and ends at the same vertex.

The  {\it evolution algebra}  $\AA_{\Gamma}$ {\it associated  to the graph} $\Gamma$ is the evolution $\mathbb{K}$-algebra  with natural basis $B=\{e_i\}_{i=1}^n$ and structure matrix $M_B$ equal to the adjacency matrix $A=(a_{ij})$ of $\Gamma$, i.e., with  relations $e_i^2 = \sum_{k=1}^n a_{ik} e_k, \, \text{for  } i=1, \dots, n$, $a_{ik} \in \mathbb{K}$, and $e_ie_j=0$ for $i \neq j$. 

\begin{exa}\label{ex:1P3} 
 The evolution $\mathbb{C}$-algebra  $\AA_{P_3}$ associated to $P_3$  graph (see Fig. \ref{FIG:P3}) has natural basis $B=\{e_1,e_2,e_3\}$, non-zero products $e_1^2=e_2, \, e_2^2=e_1+e_3, \, e_3^2=e_2$ and structure matrix: 
\begin{equation*}
M_B= \begin{pmatrix}
    0 & 1 & 0 \\
    1 & 0 & 1 \\
    0 & 1 & 0
\end{pmatrix}.
\end{equation*}
\end{exa}

The {\it evolution algebra associated  to the symmetric random walk (SRW) on} $\Gamma$ is the evolution $\mathbb{K}$-algebra, denoted by  $\AA_{RW(\Gamma)}$,  with natural basis $B=\{e_i\}_{i=1}^n$ and  structure matrix  $M_B$ equal to the transition matrix $P=(p_{ij})$ of $\Gamma$, i.e.,
 with  relations $e_i^2=\sum_{k=1}^n p_{ik} e_k, \, \text{for  } i=1, \dots, n$, $p_{ik} \in \mathbb{K}$, and $e_ie_j=0$ for $i \neq j$. 

 \begin{exa}\label{ex:SRW}
The evolution $\mathbb{C}$-algebra  $\AA_{RW(P_3)}$ associated  to the symmetric random walk on $P_3$ graph (see Fig. \ref{FIG:P3}) has natural basis $B=\{e_1,e_2,e_3\}$, non-zero products  $e_1^2=e_2, \, e_2^2=\frac{1}{2}(e_1+e_3), \, e_3^2=e_2$ and structure matrix: 
\begin{equation*}
M_B= \begin{pmatrix}
    0 & 1 & 0 \\
    \frac{1}{2} & 0 & \frac{1}{2} \\
    0 & 1 & 0
\end{pmatrix}.
\end{equation*}
\end{exa}
For $i,j \in V$, we say that $i$ is a {\it neighbor} of $j$, and we write $i \in \mathcal{N}(j)$, if $\{i,j\}\in E$. Thus, 
\begin{equation*} \label{eq:vec}
i \in \mathcal{N}(j) \textrm{ if and only if } j \in \mathcal{N}(i).
\end{equation*}
Since $\Gamma$ is connected, the evolution algebras $\AA_{\Gamma}$ and $\AA_{RW(\Gamma)}$,  with natural basis $B=\{e_i\}_{i=1}^n$,  are non-degenerate and  we have respectively: 
\begin{equation}\label{e2-grafo} e_i^2 = \sum \limits _{l \in \mathcal{N}(i)} e_l \, \hspace{0.3cm} \text{ and } \hspace{0.3cm} e_i^2=\frac{1}{d_i} \sum \limits _{l \in \mathcal{N}(i)} e_l, \, \hspace{0.3cm} \text{for  } i=1, \dots, n,
\end{equation} where $d_i$ is the  degree of the vertex $i$.

A permutation $\sigma \in S_n$ is a {\it symmetry} of $\Gamma$ if
\begin{equation}\label{symgraph}
i \in \mathcal{N}(j) \textrm{ if and only if } \sigma(i) \in \mathcal{N}(\sigma(j)),
\end{equation} 
we  denote by $\sym (\Gamma)$ the {\it symmetry group of} $\Gamma$ .

For an evolution algebra $\mathcal{A}$, we will denote by $\mathcal{L}(\mathcal{A})$ the set of linear transformations of $\mathcal{A}$ into $\mathcal{A}$,   and its  {\it automorphism group} by $\aut (\AA)$, i.e.,  
 $$\aut (\AA):=\{ f \in \mathcal{L}(\mathcal{A}) \, \, \big| \, f \, \textrm{is isomorphism of evolutions $\mathbb{K}$-algebras}\}.$$ If $B$ is a natural basis, we will denote by $[f]_{B}$ the {\it representation matrix of} $f$ {\it in the basis} $B$. Thus, if $f \in \aut (\AA)$ then $det([f]_{B}) \neq 0$ and the following condition holds for all $e_i,e_j \in B$: 
\begin{equation}\label{eq:condf}
f(e_i e_j) = f(e_i) f(e_j).\\[.5cm]
\end{equation}

\section{Results}\label{Res}
Henceforth, we will fix the following notation:
 $\Gamma=(V,E)$ for an  undirected, simple, connected and  without loops finite graph with $n$ vertices,
 $\AA \in \{\AA_{\Gamma},\AA_{RW(\Gamma)}\}$ for  the associated evolution $\mathbb{K}$-algebras
with natural basis $B=\{e_i\}_{i=1}^n$ and structure matrix $M_B$.
Since $M_B$ can be singular, as seen in  examples in \S \ref{Prem},  the following results  also apply to cases where $\AA$ is not perfect. 
\subsection{Properties of  automorphisms of evolution algebras associated to graphs}\label{both}
This section is dedicated to describe the properties of the matrix entries of an automorphism, on a natural basis. They are a consequence of specific linear combinations, which is the subject of the following lemma. 

\begin{lemma}\label{lem:intsum1}
Let $B=\{e_i\}_{i=1}^n$ be the natural basis for $\AA \in \{\AA_{\Gamma},\AA_{RW(\Gamma)}\}$.  
The following equalities hold:
\begin{enumerate} [noitemsep, leftmargin=20pt]
\item \label{sum1} for $i,\ell,k \in V$ and  $a_{k},b_{k} \in \mathbb{K}$:
\begin{equation}\label{lem:int1}
\sum_{k=1}^{n} b_{k} \left( \sum_{\ell \in \mathcal{N}(k)} a_{k} e_{\ell}\right)= \sum_{i=1}^{n}  \left( \sum_{k \in \mathcal{N}(i)} a_{k} b_{k} \right)e_{i},
\end{equation}
\item \label{sum2} for $i,\ell,k \in V$, $a_i,b_{\ell k} \in \KK$ and $\mathcal{N}(i)=\{i_{1}, i_{2}, \dots, i_{t_i}\} \subseteq V$
\begin{equation}\label{lem:int2}
\sum_{\ell \in \mathcal{N}(i)} a_{i} \left( \sum_{k=1}^{n} b_{\ell k} e_{k}\right)= \sum_{k=1}^{n}  \left( \sum_{\ell \in \mathcal{N}(i)} a_{i}b_{\ell k}  \right)e_{k}.
\end{equation}
\end{enumerate}
\end{lemma}
\begin{proof}
For {\it(1)} consider $k \in V$ and $a_{k},b_{k} \in \mathbb{K}$, then  
\begin{equation*}
\sum_{k=1}^{n} b_{k} \left( \sum_{\ell \in \mathcal{N}(k)} a_{k} e_{\ell}\right)= b_{1} \left( \sum_{\ell \in \mathcal{N}(1)} a_{1} e_{\ell}\right)+b_{2} \left(\sum_{\ell \in \mathcal{N}(2)} a_{2} e_{\ell}\right)+\dots +b_{n} \left( \sum_{\ell \in \mathcal{N}(n)} a_{n} e_{\ell}\right)
\end{equation*}
\begin{equation*}
\hspace{3.3cm} =  \left( \sum_{k \in \mathcal{N}(1)} b_{k}a_{k} \right)e_{1}+ \left( \sum_{k \in \mathcal{N}(2)} b_{k}a_{k} \right)e_{2}+ \dots + \left( \sum_{k \in \mathcal{N}(n)} b_{k}a_{k} \right)e_{n}
\end{equation*}\vspace{0.1cm}
\begin{equation*}
\hspace{-3.1cm} =   \sum_{i=1}^{n} \left( \sum_{k \in \mathcal{N}(i)} b_{k}a_{k} \right)e_{i}.
\end{equation*}
For {\it(2)} consider $i,\ell,k \in V$ and $a_i,b_{\ell k} \in \KK$. If $\mathcal{N}(i)=\{ i_{1}, \dots, i_{t_i}\}$ then,  
\begin{equation*}
\sum_{\ell \in \mathcal{N}(i)} a_{i} \left( \sum_{k=1}^{n} b_{\ell k} e_{k}\right)= a_{i} \left(\sum_{k=1}^{n} b_{i_1 k} e_{k} \right)+ a_{i}\left(\sum_{k=1}^{n} b_{i_{2} k} e_{k} \right)+ \dots + a_{i} \left( \sum_{k=1}^{n} b_{i_{t_i} k} e_{k} \right)
\end{equation*}
\begin{equation*}
\hspace{4.5cm} =  a_{i} \left( \sum_{\ell \in \mathcal{N}(i)} b_{\ell 1} \right)e_{1}+a_{i} \left( \sum_{\ell \in \mathcal{N}(i)} b_{\ell 2} \right)e_{2}+\dots +a_{i} \left(\sum_{\ell \in \mathcal{N}(i)} b_{\ell n} \right)e_{n}
\end{equation*}
\begin{equation*}
\hspace{-2.1cm} =   \sum_{k=1}^{n} a_{i} \left( \sum_{\ell \in \mathcal{N}(i)} b_{\ell k} \right)e_{k}. 
\end{equation*}
\end{proof}

\begin{remark}\label{rem:int}
Notice that if $\sigma \in S_n$, then \eqref{lem:int1} and  \eqref{lem:int2} are fulfilled as follows, respectively
\begin{equation}\label{rem:int1}
\sum_{k=1}^{n} b_{k} \left( \sum_{\ell \in \mathcal{N}(k)} a_{k} e_{\sigma(\ell)}\right)= \sum_{i=1}^{n}  \left( \sum_{k \in \mathcal{N}(i)} a_{k} b_{k} \right)e_{\sigma(i)} \vspace{0.3cm}
\end{equation}
and 
\begin{equation}\label{rem:int2}
\sum_{\ell \in \mathcal{N}(i)} a_{i} \left(\sum_{k=1}^{n} b_{\ell k} e_{\sigma(k)}\right)= \sum_{k=1}^{n}  \left( \sum_{\ell \in \mathcal{N}(i)} a_{i}b_{\ell k}  \right)e_{\sigma(k)}.
\end{equation}
\end{remark} \vspace{0.1cm}

\begin{prop}\label{eqayba1a1} 
Let $B=\{e_i\}_{i=1}^n$ be the natural basis for $\AA_{\Gamma}$ or $\AA_{RW(\Gamma)}$, $i,k\in V$ and  $t_{ik} \in \KK$.
\begin{enumerate}[noitemsep, leftmargin=20pt]
\item If $f \in \aut (\AA_{\Gamma})$ is such that 
$f(e_i) =\sum_{k=1}^n t_{ik}e_k$ then,\vspace{0.2cm}
\begin{equation}\label{eq:propa1}
\sum_{k\in \mathcal{N}(r)}t_{ik}t_{jk} =0,\text{ for all }r,i,j\in V\text{ such that }i\neq j,
\end{equation}
and
\begin{equation}\label{eq:propb1}
\sum_{k\in \mathcal{N}(r)} t_{ik}^2 = \sum_{\ell \in \mathcal{N}(i)}t_{\ell r},\text{ for all }i,r\in V.
\end{equation}\vspace{0.2cm}
\item If $f \in \aut (\AA_{RW(\Gamma)})$ such that 
$f(e_i) =\sum_{k=1}^n t_{ik}e_k$  then,\vspace{0.2cm}
\begin{equation}\label{RWeq:propa2}
\sum_{k\in \mathcal{N}(r)} \frac{1}{d_k}t_{ik}t_{jk} =0,\text{ for all }r\in V\text{ and }i,j\in V\text{ such that }i\neq j,
\end{equation}
and
\begin{equation}\label{RWeq:propb2}
\sum_{k\in \mathcal{N}(r)} \frac{1}{d_k}t_{ik}^2 = \sum_{\ell \in \mathcal{N}(i)} \frac{1}{d_i}t_{\ell r},\text{ for all }i,r\in V.
\end{equation}\vspace{0.2cm}
\end{enumerate}
\end{prop}
\begin{proof}
For {\it (1)} consider $i,j\in V$ with $i\neq j$. Since $e_i e_j =0$ we have from \eqref{eq:condf} that
\begin{equation*}
0=f(e_i e_j )=f(e_i) f(e_j)=\left(\sum_{k=1}^n t_{ik}e_k\right)\left(\sum_{k=1}^n t_{jk}e_k\right)=\sum_{k=1}^n t_{ik}t_{jk}e_k^2
\end{equation*}
\begin{equation*}\label{prodceros2}
\hspace{-2.7cm}=\sum_{k=1}^n t_{ik}t_{jk}\left(\sum_{l \in \mathcal{N}(k)} e_l\right)=\sum_{s=1}^n \left(\sum_{k\in \mathcal{N}(s)} t_{ik}t_{jk}\right)e_s,
\end{equation*}
 where the last equalities come  from (\ref{e2-grafo})  and  Lemma \ref{lem:intsum1} {\it (1)} respectively, 
which proves \eqref{eq:propa1}. 

Now, observe that:
\begin{equation}\label{eq:cuad3}
f(e_i^2)=f\left(\sum_{\ell \in \mathcal{N}(i)} e_{\ell}\right)=\sum_{\ell \in \mathcal{N}(i)} f(e_{\ell})=\sum_{\ell \in \mathcal{N}(i)} \left(\sum_{k=1}^n  t_{\ell k}e_{k} \right)=\sum_{s=1}^n \left(\sum_{\ell \in \mathcal{N}(i)} t_{\ell s}\right)e_s,
\end{equation}
 the last equality comes from Lemma \ref{lem:intsum1} {\it (2)}, while
\begin{equation}\label{eq:cuad4}
f(e_i)^2 = \left(\sum_{k=1}^n t_{ik}e_k\right)^2 = \sum_{k=1}^n t_{ik}^2 e_k^2 = \sum_{k=1}^n t_{ik}^2\left( \sum_{\ell \in \mathcal{N}(k)}e_{\ell}\right)=\sum_{s=1}^n \left(\sum_{ k \in \mathcal{N}(s)} t_{ik}^2\right) e_s,
\end{equation}
where the last equalities come  from (\ref{e2-grafo})  and  Lemma \ref{lem:intsum1} {\it (1)}, respectively. Therefore \eqref{eq:propb1} is a consequence of \eqref{eq:condf}, \eqref{eq:cuad3} and \eqref{eq:cuad4}.

The proof for {\it (2)} is analogous, just consider  $\AA_{RW(\Gamma)}$, for $i,j\in V$ with $i\neq j$. 
\begin{equation*}
0=f(e_i e_j )=f(e_i) f(e_j)
=\sum_{k=1}^nt_{ik}t_{jk}e_k^2
=\sum_{k=1}^n t_{ik}t_{jk}\left( \sum_{l \in \mathcal{N}(k)} \frac{1}{d_k} e_l\right) 
=\sum_{s=1}^n \left(\sum_{k\in \mathcal{N}(s)} \frac{1}{d_k}t_{ik}t_{jk}\right)e_s,
\end{equation*} and 
\begin{equation*}\label{RWeq:cuad5}
f(e_i^2)=f\left( \frac{1}{d_i} \sum_{\ell \in \mathcal{N}(i)} e_{\ell}\right)
=\sum_{\ell \in \mathcal{N}(i)} \frac{1}{d_i} \left(\sum_{k=1}^n  t_{\ell k} e_k\right)=\sum_{s=1}^n\left(\sum_{\ell \in \mathcal{N}(i)} \frac{1}{d_i} t_{\ell s} \right)e_s,
\end{equation*}
while
\begin{equation*}\label{RWeq:cuad6}
f(e_i)^2 = 
\sum_{k=1}^n t_{ik}^2 e_k^2 = \sum_{k=1}^n t_{ik}^2 \left( \frac{1}{d_k} \sum_{l\in \mathcal{N}(k)}  e_{l}\right)=\sum_{s=1}^n \left(\sum_{k\in \mathcal{N}(s)} \frac{1}{d_k}t_{ik}^2\right) e_s.
\end{equation*}\vspace{0.2cm}
\end{proof}

The proof of the following corollary follows directly from \eqref{eq:propa1} to  \eqref{RWeq:propb2}.

\begin{corollary}\label{GemyColg}
    Under the same conditions of the previous proposition, it holds:
    \begin{enumerate}[noitemsep, leftmargin=20pt]
     \item 
        for  $i,j,k \in V,$ if $ \mathcal{N}(j)= \mathcal{N}(k)$, then
    \begin{equation} \label{gemelos1}
        \sum_{\ell \in \mathcal{N}(i)}t_{\ell k}=\sum_{\ell \in \mathcal{N}(i)}t_{\ell j,}  \end{equation}
        and
        \begin{equation} \label{gemelos2} \textrm{either }\; \sum_{\ell \in \mathcal{N}(i)}t_{k \ell }^2=\sum_{\ell \in \mathcal{N}(i)}t_{j \ell}^2 \hspace{0.5cm} 
     \textrm{or} \hspace{0.5cm} \sum_{l \in \mathcal{N}(i)}\frac{1}{d_l}t_{k l }^2=\sum_{l \in \mathcal{N}(i)}\frac{1}{d_l}t_{j l}^2 ,
    \end{equation}
    it depends on whether $f \in \aut (\AA_{\Gamma})$ or $f \in \aut (\AA_{RW(\Gamma)}).$
     \item  If $\ell\in V$ is such that   $\mathcal{N}(\ell)=\{k_{\ell}\}$, then 
    \begin{equation} \label{colg1}
    \displaystyle t_{i\, k_{\ell}}t_{j\,k_{\ell}}= 0, \hspace{0.2cm} \textrm{for all} \hspace{0.2cm} i,j\in V, \hspace{0.2cm} \textrm{with}  \hspace{0.2cm} i\neq j
    \end{equation}   
and 
\begin{equation} \label{colg2}
    \displaystyle t_{i\,k_{\ell}}\neq 0, \hspace{0.2cm} \textrm{for exactly one} \hspace{0.2cm} i\in V.
     \end{equation}
      \end{enumerate}
  \end{corollary}

\subsection{Diagonal subgroup of \boldmath{$\aut (\AA)$.}} \label{dgn}

Let $\AA \in \{\AA_{\Gamma},\AA_{RW(\Gamma)}\}$ and  $\DD \leq \aut(\AA)$ be the subgroup of automorphisms $f$ such that $f(e_i) = t_{ii}e_i$, for $t_{ii} \in \mathbb{K}^*$ and  $i\in V$, i.e.,
\begin{equation*}
[f]_{B}= \begin{pmatrix}
    t_{11} & 0 & \dots & 0 \\
   0 & t_{22} & \dots & 0 \\
    \vdots & \vdots& \ddots & \vdots\\
    0 & 0 & \dots & t_{nn}
\end{pmatrix}.
\end{equation*}
Let us denote by $(t_{11}, t_{22}, \dots, t_{nn})_{[f]_B}$ its diagonal and, from now on,  for $f \in \DD$,  we will identify $[f]_{B}$ with $(t_{11}, t_{22}, \dots, t_{nn})_{[f]_B}$. We  call $\DD$ the {\it diagonal subgroup of $\aut (\AA)$ relative to $B$}. Since our basis is fixed and given by the graph $\Gamma$, we call $\DD$  simply the {\it diagonal subgroup of $\aut (\AA)$}.

In \cite{Eld-Lab2015, Eld-Lab2021}, assuming that $\AA$ is perfect, it has been shown that the diagonal subgroup of $\aut (\AA)$ is isomorphic to the  group:
$$\{(u_1, \dots, u_n) \in (\mathbb{K}^*)^n \, | \, u_i^2=u_{\ell},   (i,\ell)\in E\},$$
{\it the diagonal subgroup of $\Gamma$ in $\mathbb{K}$}.   The next proposition shows that, in our context, these groups are isomorphic even if the algebras are not perfect. Since our graphs are bi-directed recall that
$\{i,\ell\}\in E$ if and only if $\ell \in \mathcal{N}(i)$ if and only if $ i \in \mathcal{N}(\ell)$.  

\begin{prop} \label{eqaybDiag} 
Let $B=\{e_i\}_{i=1}^n$ be the natural basis for $\AA \in \{\AA_{\Gamma},\AA_{RW(\Gamma)}\}$,  $f \in \mathcal{L}(\AA)$  and $\mathcal{D} \leq \aut (\AA)$. Then  $f \in \mathcal{D}$ if and only if $[f]_{B}$  is a diagonal matrix such that: 
\begin{enumerate}[noitemsep, leftmargin=20pt]
 \item $t_{ii} \in \mathbb{K}^*$ 
for all $i\in V$,
    \item $t_{ii}^2=t_{\ell \ell}$ for all $\ell \in \mathcal{N}(i)$,
    \item for all $i \in V$, $t_{ii}$ is a cubic root of the unit of $\mathbb{K}$.
   \end{enumerate}
\end{prop}
\begin{proof}
 If $f \in \mathcal{D} $ then: 
\begin{enumerate}[noitemsep, leftmargin=20pt]
\item[{\it (1)}] By definition
$f(e_i) = t_{ii}e_i$, for $t_{ii} \in \mathbb{K}^*$ 
for all $i\in V$. 
    \item[{\it (2)}] For $i\in V$, 
    \begin{equation*}
    f(e_i^2)=      \begin{cases}
        \sum_{l \in \mathcal{N}(i)} t_{ll}e_{l}, \textrm { if $\AA=\AA_{\Gamma}$},\\
     \sum_{l \in \mathcal{N}(i)} \frac{1}{d_i} t_{ll}e_{l}, \textrm { if  $\AA=\AA_{RW(\Gamma)}$}.
 \end{cases}
    \end{equation*}
While
\begin{equation*}
f(e_i)^2 = t_{ii}^2e_i^2 =\begin{cases} \sum_{l\in \mathcal{N}(i)}t_{ii}^2e_{l}, \textrm { if  if $\AA=\AA_{\Gamma}$},\\
\sum_{l\in \mathcal{N}(i)}\frac{1}{d_i} t_{ii}^2e_{l}, \textrm { if   if  $\AA=\AA_{RW(\Gamma)}$}.
\end{cases}
\end{equation*}
Then $t_{ii}^2=t_{ll}$ for all $l \in \mathcal{N}(i)$. \vspace{0.3cm}
\item[{\it (3)}] Let $i, \ell \in V$ such that  $\ell \in \mathcal{N}(i)$. Since $\ell \in \mathcal{N}(i)$, if and only if, $i \in \mathcal{N}(\ell)$, then from the previous item we have:  $t_{ii}^2=t_{\ell \ell}$ for all $\ell \in \mathcal{N}(i)$ and  $t_{\ell\ell}^2=t_{ii}$ for all $i \in \mathcal{N}(\ell)$. Hence $t_{ii}^4=t_{ii}$ and so  $t_{ii}^3=1$, which shows that for all $i \in V$, $t_{ii}$ is a cubic root of the unit of $\mathbb{K}$.
     \end{enumerate}
     
If $[f]_{B}$  is a diagonal matrix, the reciprocal is immediate from  {\it (1), (2), (3).}
\end{proof} 

\begin{corollary} Let $\mathbb{K}$ be an algebraically closed field.
Under the assumptions of Proposition \ref{eqaybDiag}, if  $1,\alpha, \beta$ denote the three distinct cubic roots of the unity  and  $(t_{11}, t_{22}, \dots, t_{nn})_{[f]_B}$ is the diagonal of $f\in \DD$ then:
\begin{enumerate} [noitemsep, leftmargin=25pt]
\item  If $t_{ii}=1$ for some $i\in V,$ then $f=id$.
\item If  $t_{ii}=\alpha$ for some $i$, then there exists $j\in V\setminus \{i\}$ such that $t_{jj}=\beta$.
\item If $f \neq Id$, then for all $i\in V$ and all $j\in \mathcal{N}(i)$ we have $t_{ii}\neq t_{jj}$.
\end{enumerate}
\end{corollary}
\begin{proof}
{\it (1)} If $t_{ii}=1$ for some $i\in V$, then from Proposition \ref{eqaybDiag}, $t_{\ell \ell}=1$ for all $\ell \in \mathcal{N}(i)$. Thus,   for  the connectedness of the graph, $t_{jj}=1$ for all $j \in V$ and $f=id$.\\
{\it (2)} If  $t_{ii}=\alpha$ for all $i \in V$, then from Proposition
\ref{eqaybDiag} {\it (2)}, $\alpha^2=\alpha$ and $\alpha=1$, a contradiction with the hypothesis.\\
{\it (3)} If $f \neq Id$ and there exist $i\in V$ and  $j\in \mathcal{N}(i)$ with $t_{ii}= t_{jj}$,  from Proposition \ref{eqaybDiag} {\it (2)} we have $t_{ii}=1$. Then from item {\it (1)} above, $f=id$ and we have a contradiction.
\end{proof}

\begin{remark}\label{bipartito}
From the previous results, it is clear that the possibility of having $\DD \neq \{Id\}$ depends on the existence of a non-trivial partition of $V$ in two sets, say $V_1$ and $V_2$,  in such a way that:
\begin{itemize}[noitemsep, leftmargin=20pt]
\item we can identify the elements in $V_1$ with the same non-trivial cubic root of the unity and the elements in  $V_2$ with the other,
\item  $i\in V_1$, if and only if, $\mathcal{N}(i) \subseteq V_2.$ But, since $V_1 \cap V_2 = \emptyset$,  this is exactly the definition of a bipartite graph, (see  \cite{BR}).
\end{itemize}
\end{remark}

Recall that a cycle in a graph is an edge path that starts and ends at the same vertex. Since we are dealing with  bidirected graphs,   every vertex is in a cycle.

\begin{lemma}[Theorem 1.5.10 \cite {BR}]\label{oddcycle}
 A graph is bipartite if and only if it contains no odd-length cycles.   
\end{lemma} 

\begin{theorem}\label{D1o3}Let $\AA \in \{\AA_{\Gamma},\AA_{RW(\Gamma)}\}$  an evolution $\mathbb{K}$-algebra associated to $\Gamma$, for $\mathbb{K}$ algebraically closed.  If $\mathcal{D} \leq \aut (\AA)$ is the diagonal subgroup, then 
 $\DD \neq \{id\}$  if and only if  $\Gamma$  is  bipartite.
\end{theorem}
\begin{proof}
 Without loss of generality,  fix $1 \in V$, $2 \in \mathcal{N}(1)$ and, recalling that $l(\pi)$ denotes the length of a path $\pi$, we define the following two subsets of $V$:
\begin{equation*}\label{V1V2}
\begin{array}{ccc}
V_1 &:=\{i \in V|\;   \textrm{ there is a path $\pi$ between $1$ and $i$ such that $l(\pi)$ is even}\}, \\
V_2 &:=\{i \in V|\;  \textrm{ there is a path $\pi$ between $1$ and $i$ such that $l(\pi)$ is odd}\}.
\end{array}
\end{equation*}
Notice that $1 \in V_1$, $2 \in V_2 $ and $V_1 \cup V_2=V$. Moreover, $V_1 \cap V_2 \neq \emptyset$ if and only if the vertex $1$ is in an odd length cycle of $\Gamma$. In this case, due to the connectedness of $\Gamma$,  every vertex in the graph  is in an odd length cycle.

If we suppose that $\Gamma$ is a bipartite graph, then from Lemma \ref{oddcycle},  $V_1 \cap V_2 = \emptyset.$ Hence, if $\alpha, \beta$ are the non-trivial roots of the unity, we could define 
\begin{equation}\label{fs}
f(e_i):= 
\begin{cases}
\alpha e_i, \textrm{ if $i \in V_1$},\\
\beta e_i, \textrm{ if $i \in V_2$}.
\end{cases}
\end{equation}
It is clear that $f \neq id$ and we claim that $f \in \DD$. In fact,
notice that $f$ satisfies items {\it (1)} and {\it (3)}  of Proposition \ref{eqaybDiag}. For item {\it (2)}, since $\alpha^2 =\beta$, it is enough to show that for all $ i \in V$ and for all $ \ell \in \mathcal{N}(i) $ then $i \in V_1$ if and only if $\ell \in V_2$. Indeed,  since 
$V_1 \cap V_2= \emptyset$, then  $i \in V_1$ if and only if there exists an even length  path between $1$ and $i$. Equivalently,  for all $\ell \in \mathcal{N}(i)$ there is an odd length path between $1$ and $\ell$, because $d(\ell, i)=1$, see \eqref{distance}. Equivalently, $\ell \in V_2$, for all $\ell \in \mathcal{N}(i)$.  

The other implication follows from Remark \ref{bipartito}.
\end{proof}

\begin{corollary}\label{AlgD}
      With the hypothesis  of the  theorem above, either $\DD=\{Id\}$  or  $\DD \cong \mathbb{Z}_3$.   \end{corollary}
      \begin{proof}
      If $\DD$ is not trivial, from the arguments in proof of Theorem \ref{D1o3}, in \eqref{fs}, there are two nontrivial possibilities for $f \in \DD:$ 
      \begin{equation*}
 \hspace*{-4.5cm}     \textrm{either }\quad
f(e_i):= 
\begin{cases}
\alpha e_i, \textrm{ if $i \in V_1$},\\
\beta e_i, \textrm{ if $i \in V_2$}
\end{cases}\,
\textrm { or } \quad
f(e_i):= 
\begin{cases}
\beta e_i, \textrm{ if $i \in V_1$},\\
\alpha e_i, \textrm{ if $i \in V_2$}.
\end{cases}
\end{equation*}
Then $|\DD|=3$ and therefore  $\DD \cong \mathbb{Z}_3.$
      \end{proof}

 Corollary \ref{AlgD} makes it possible to calculate quickly and easily all the elements that belong to the diagonal subgroup $\mathcal{D} $ of $\aut (\AA_{\Gamma})$, for the evolution algebra $\AA_{\Gamma}$, associated to any finite graph undirected, simple, connected, and without loops. 
     We illustrate this process for both, perfect and non-perfect evolution algebras, in the following examples.
     
  \begin{exa} \label{ex:autC22}
 Let $\AA_{C_{2,2}}$ be the evolution $\mathbb{C}$-algebra associated to the caterpillar tree graph $C_{2,2}$ (see Fig. \ref{FIG:C22}), which is  not perfect.
\begin{figure}[h]
\centering
\begin{tikzpicture}[scale=0.4]
\draw (3.9,0) node[above,font=\footnotesize] {$\,$} -- (2.8,2);
\draw (3.9,-1.4) node[above,font=\footnotesize] {$1$};
\draw (2.8,2) node[above,font=\footnotesize] {$2$};
\draw (4.9,2) node[above,font=\footnotesize] {$3$} -- (3.9,0);
\draw (9.9,0) node[above,font=\footnotesize] {$\,$} -- (3.9,0);
\draw (9.9,-1.4) node[above,font=\footnotesize] {$4$};
\draw (8.7,2) node[above,font=\footnotesize] {$5$} -- (9.9,0);
\draw (10.9,2) node[above,font=\footnotesize] {$6$} -- (9.9,0);
\filldraw [black] (3.9,0) circle (3.5pt);
\filldraw [black] (2.8,2) circle (3.5pt);
\filldraw [black] (4.9,2) circle (3.5pt);
\filldraw [black] (9.9,0) circle (3.5pt);
\filldraw [black] (8.7,2) circle (3.5pt);
\filldraw [black] (10.9,2) circle (3.5pt);
\end{tikzpicture}
\caption{ Caterpillar tree graph $C_{2,2}$}
\label{FIG:C22}
\end{figure} 
Since $C_{2,2}$ is  bipartite, then 
$ \DD \neq \{Id\}$. Thus, from Corollary \ref{AlgD}, if $V_{1}=\{1,5,6\}$, $V_{2}=\{2,3,4\}$ and \,$1,  \alpha=\frac{-1+\sqrt{3}i}{2}, \beta =\frac{-1-\sqrt{3}i}{2}$ are the three cubic roots of the unit in $\mathbb{C}$, we have $\DD= \{Id, h,g \}$     with $ (1,1,1,1,1,1)_{[Id]_B}, \, (\alpha, \beta,\beta,\beta,\alpha,\alpha)_{[h]_B}$ and $ (\beta,\alpha,\alpha,\alpha,\beta,\beta)_{[g]_B}$.
\end{exa} 

\begin{exa}\label{ex:DP4} Consider the perfect evolution $\mathbb{C}$-algebra $\AA_{P_{4}}$ for the path graph $P_4$   (see Fig. \ref{FIG:DP4}). Since $P_4$ is  bipartite, Theorem \ref{D1o3} implies that $D \neq \{Id\}$. Then, by Corollary \ref{AlgD}, if $V_{1}=\{1,3\}$, $V_{2}=\{2,4\}$ and \,$1, \alpha, \beta$  are the three cubic roots of the unit in $\mathbb{C}$, then  $\DD= \{Id,h,g\}$     with $ (1,1,1,1)_{[Id]_B}, \, (\alpha, \beta,\alpha,\beta)_{[h]_B}$ and $ (\beta,\alpha,\beta,\alpha)_{[g]_B}$. 

\begin{figure}[h!]
\begin{center}
\begin{tikzpicture}[scale=0.4]
\draw (3.9,0) node[above,font=\footnotesize] {$1$} -- (6.9,0);
\draw (6.9,0) node[above,font=\footnotesize] {$2$} -- (9.9,0);
\draw (9.9,0) node[above,font=\footnotesize] {$3$} -- (12.8,0) node[above,font=\footnotesize] {$4$};
\filldraw [black] (3.9,0) circle (3.5pt);
\filldraw [black] (6.9,0) circle (3.5pt);
\filldraw [black] (9.9,0) circle (3.5pt);
\filldraw [black] (12.9,0) circle (3.5pt);

\hspace{-1.2cm}

\end{tikzpicture}
				\caption{ Path graph $P_4$.}
                 \label{FIG:DP4}
			\end{center}
\end{figure}

   \end{exa}

   \begin{remark}
   Although we have been working on algebraically closed fields of characteristic zero, we can use the same method as in the previous examples to determine the diagonal group over other fields. For example,  in Example \ref{ex:DP4}, if $\mathbb{K}=\mathbb{R}$ or $\mathbb{K}=\mathbb{Z}_3$, then $t_{ii}=1_{\mathbb{K}}$ for all $i$ and  therefore    $\DD= \{Id\}.$
   If $\mathbb{K}=\mathbb{Z}_7$, then $\bar{1},\bar{2}, \bar{4}$ are the cubic roots of the unit and therefore   $\DD= \{Id,h,g\}$
    with $(\bar{1},\bar{1},\bar{1},\bar{1})_{[Id]_B}, \, (\bar{2},\bar{4},\bar{2},\bar{4})_{[h]_B}$ and $ (\bar{4},\bar{2},\bar{4},\bar{2})_{[g]_B}$.
 Thus, the diagonal subgroup depends on the field we are working in, if it is algebraically closed or not.
    \end{remark}
 \subsection{\boldmath{$\mathbf{\sym(\Gamma)}$} acts on $\mathbf{\aut (\AA)}$}\label{acts}
In this section, we will describe 
the set 
$\aut (\mathcal{A})$, for $\AA \in \{\AA_{\Gamma},\AA_{RW(\Gamma)}\}$,
via a free
$\sym(\Gamma)$-action. 
This will allow us to extend an important result about perfect evolution algebras, in \cite{Eld-Lab2015, Eld-Lab2021} and revisited in \cite{CLTV, Sri-Zou2022}, to a family of evolution algebras, not necessarily perfect, in the class of evolution algebras associated to graphs.

\begin{prop}\label{prop:gen1}
Let $\sigma \in \sym(\Gamma)$ and $f \in \aut (\AA)$, with  $f(e_i)=\sum_{k=1}^n t_{ik}e_{k}$ for $t_{ik} \in \mathbb{K}$. If $g(e_i):=\sum_{k=1}^n t_{ik}e_{\sigma(k)}$ then $g \in \aut (\AA)$.
\end{prop}
\begin{proof}
If  $i\neq j$, then $g(e_i e_j)=0$. On the other hand: \vspace{-0.2cm}
$$g(e_i)  g(e_j)=\left(\sum_{k=1}^n t_{ik}e_{\sigma(k)}\right)\left(\sum_{k=1}^n t_{jk}e_{\sigma(k)}\right)  =  \sum_{k=1}^n t_{ik}t_{jk}e_{\sigma(k)}^2 = \vspace{-0.2cm}$$
\begin{equation*} \begin{cases}
\sum \limits_{k=1}^n \left(t_{ik}t_{jk}\sum \limits_{\sigma(q)\in \mathcal{N}(\sigma(k)) }e_{\sigma(q)}\right)=\sum \limits_{k=1}^n\left(\sum \limits_{q\in \mathcal{N}(k)}  t_{ik}t_{jk}  e_{\sigma(q)}\right)  =\sum \limits_{k=1}^n\left(\sum \limits_{k\in \mathcal{N}(q)} t_{ik}t_{jk} \right)e_{\sigma(q)}, \vspace{0.2cm}  \textrm{ if $\AA=\AA_{\Gamma}$,} 
\\
\sum \limits_{k=1}^n \left(\frac{t_{ik}t_{jk}}{d_{\sigma(k)}}\sum \limits_{\sigma(q)\in \mathcal{N}(\sigma(k)) }e_{\sigma(q)}\right)=\sum \limits_{k=1}^n\left(\sum \limits_{q\in \mathcal{N}(k)} \frac{ t_{ik}t_{jk}}{d_{\sigma(k)}}  e_{\sigma(q)}\right)  =\sum \limits_{k=1}^n\left(\sum \limits_{k\in \mathcal{N}(q)} \frac{t_{ik}t_{jk}}{d_k} \right)e_{\sigma(q)}, \vspace{0.2cm} 
 \textrm{ if  $\AA=\AA_{RW(\Gamma)},$} 
\end{cases} 
\end{equation*}
$=0,$

\noindent where the last three equalities, in each case, come from \eqref{lem:int1}, (\ref{symgraph}) and either \eqref{eq:propa1} or \eqref{RWeq:propa2},  respectively, hence $g(e_i  e_j)=g(e_i)  g(e_j)$. 
Now, by  Remark \ref{rem:int}, we observe   that: 
$g(e_i^2)=$
\begin{equation*}\label{eq:propgen11}
\hspace*{-3pt}
\begin{cases}
g\left(\sum \limits_{\ell \in \mathcal{N}(i)} e_{\ell}\right)=\sum \limits_{\ell \in \mathcal{N}(i)} g(e_{\ell})=\sum \limits_{\ell \in \mathcal{N}(i)}\left(\sum \limits_{k=1}^n t_{\ell k} e_{\sigma(k)} \right) =\sum \limits_{r=1}^n\left(\sum \limits_{\ell \in \mathcal{N}(i)} t_{\ell r} \right)e_{\sigma(r)}, \textrm{ if  $\AA=\AA_{\Gamma}$,}\\ \\
g\left(\frac{1}{d_i}\sum \limits_{\ell \in \mathcal{N}(i)} e_{\ell}\right)=\sum \limits_{\ell \in \mathcal{N}(i)} \frac{1}{d_i} g(e_{\ell})=\sum \limits_{\ell \in \mathcal{N}(i)} \frac{1}{d_i} \left(\sum \limits_{k=1}^n t_{\ell k} e_{\sigma(k)}\right)=\sum \limits_{r=1}^n\left(\sum \limits_{\ell \in \mathcal{N}(i)} \frac{1}{d_i} t_{\ell r} \right)e_{\sigma(r)},\textrm{ if  $\AA=\AA_{RW(\Gamma)}.$}
\end{cases}
\end{equation*}
On the other hand, $g(e_i)^2 = \sum \limits_{r=1}^nt_{ir}^2 e_{\sigma(r)}^2 =$

\begin{equation*}\label{eq:propgen12}
\hspace*{-5pt}
\begin{cases}
  \sum \limits_{r=1}^nt_{ir}^2\left( \sum \limits_{{\sigma(q)} \in \mathcal{N}({\sigma(r)})}e_{{\sigma(q)}}\right)=\sum \limits_{r=1}^n  \left(\sum \limits_{q\in \mathcal{N}(r)} t_{ir}^2  e_{\sigma(q)}\right)= \sum \limits_{r=1}^n \left(\sum \limits_{k\in \mathcal{N}(r)} t_{ik}^2 \right) e_{\sigma(r)}, \textrm{ if  $\AA=\AA_{\Gamma}$,}\\\\
   \sum \limits_{r=1}^nt_{ir}^2\left( \sum \limits_{{\sigma(q)} \in \mathcal{N}({\sigma(r)})} \frac{1}{d_{\sigma(r)}}e_{{\sigma(q)}}\right) 
 =\sum \limits_{r=1}^n  \left(\sum \limits_{q\in \mathcal{N}(r)} \frac{1}{d_r} t_{ir}^2  e_{\sigma(q)}\right)= \sum \limits_{r=1}^n \left(\sum \limits_{k\in \mathcal{N}(r)} \frac{1}{d_r} t_{ik}^2 \right) e_{\sigma(r)}, \textrm{ if  $\AA=\AA_{RW(\Gamma)},$}
 \end{cases}
\end{equation*}
 where the last two equalities, in each case,   come from (\ref{symgraph}) and (\ref{rem:int1}). Then  by  Proposition  \ref{eqayba1a1} (in \eqref{eq:propb1} for $\AA=\AA_{\Gamma}$, in \eqref{RWeq:propb2} for $\AA=\AA_{RW(\Gamma)}$) we have $g(e_i^2)=g(e_i)^2.$
\end{proof} \vspace{0.2cm}

We will denoted by $f_{ \langle \sigma \rangle }$  the automorphism defined in  Proposition \ref{prop:gen1} 
from $f \in \aut (\AA)$ 
and $\sigma \in \sym(\Gamma)$, i.e., if $f(e_i)=\sum_{k=1}^n t_{ik}e_{k}$, then  
$$f_{ \langle \sigma \rangle }(e_i):=\sum_{k=1}^n t_{ik}e_{\sigma(k)}$$
and we will say that $f_{ \langle \sigma \rangle }$ is {\it generated} by $f$ and $\sigma$.

\begin{remark}\label{matrices2}
Notice that $[f_{\langle \sigma \rangle}]_{B}$ is obtained from $[f]_{B}$ via the column permutation defined by $\sigma$.
\end{remark}

Let ${\rm id}_{\Gamma} \in \sym(\Gamma)$ denote the identity automorphism and $\AA \in \{\AA_{\Gamma},\AA_{RW(\Gamma)}\}$.
Notice that  every  $f \in \aut (\AA)$ is generated by $f$ and ${\rm id}_{\Gamma}$. However,  we can find  $\sigma \in \sym(\Gamma) \setminus \{{\rm id}_{\Gamma}\}$ and  $g \in \aut (\AA) \setminus \{f\}$ such that $f=g_{ \langle \sigma \rangle }$. In fact, if $f(e_i)=\sum_{k=1}^n t_{ik}e_{k}$, then for all $\mu \in \sym(\Gamma)$, $g=f_{ \langle \mu \rangle } \in \aut(\AA)$ and $f=g_{ \langle \mu^{-1} \rangle} $. Hence, associate to $f\in \aut (\AA)$  we have the finite set $\{f_{ \langle \sigma \rangle }\}_{\sigma \in \sym(\Gamma)}$ with cardinality  $|\sym(\Gamma)|$, because $f_{ \langle \sigma_1 \rangle } \neq f_{ \langle \sigma_2 \rangle }$, as long as, $\sigma_1 \neq \sigma_2$. 

\begin{exa}
  For the path graph $P_3$ (see Fig. \ref{FIG:P3(a)}), it is clear that $\mu=(12) \notin  \sym(\Gamma)$. For $Id \in \aut (\AA_{\Gamma})$, if $g:=Id_{ \langle \mu \rangle }$ then $g \notin \aut(\AA_{\Gamma})$ since $g(e_3^2)=e_1$ while $g(e_3)^2=e_2$.
\end{exa}

The example above shows the importance of the Proposition \ref{prop:gen1} for the good definition of the action  in the next theorem.

\begin{theorem}\label{action} 
 $ \sym(\Gamma)$ acts freely on  $ \aut (\AA)$ by
$$\begin{array}{cccclll}
 & \sym(\Gamma) \times \aut(\AA) & \longrightarrow & \aut(\AA)        &     &                 &  \vspace{0.2cm} \\
         & (\sigma,f) & \longmapsto          & \sigma.f:=f_{ \langle \sigma \rangle }
\end{array} $$ 
\begin{proof} 
It is an action: 
\begin{enumerate}[noitemsep, leftmargin=20pt]
\item  Let ${\rm id}_{\Gamma} \in \sym(\Gamma) $ denote the identity automorphism. Then for all  $f \in \aut(\AA) $ we have ${\rm id}_{\Gamma}.f=f_{ \langle {\rm id}_{\Gamma} \rangle }=f$.
\item Let $\gamma, \sigma \in \sym(\Gamma)$ and $f \in \aut(\AA) $, with $f(e_i)=\sum_{k=1}^n t_{ik}e_{k}$. Then  $$((\gamma \circ \sigma).f)(e_i)=f_{ \langle \gamma \circ \sigma \rangle }(e_i)=\sum_{k=1}^n t_{ik}e_{(\gamma \circ \sigma)(k)}=\sum_{k=1}^n t_{ik}e_{\gamma (\sigma(k))}=(\sigma . f)_{\langle \gamma \rangle}(e_i)=\gamma.(\sigma.f)(e_i).$$ 
\end{enumerate}
It is free:  suppose that there exists $\gamma  \in \sym(\Gamma)$ such that  $f_{\langle \gamma \rangle}=f$. Then, $[f_{\langle \gamma \rangle}]_B=[f]_B$ and this implies $\gamma={\rm id}_{\Gamma}$. In a contrary case,  there are in $[f]_B$ at least two equal columns, but this is a contradiction with the non-singularity of $[f]_B$.
    \end{proof} 
\end{theorem}


We will denote by $O_{f}$ the orbit of $f \in \aut(\AA)$  by the action, i.e.,
$$O_f:=\{f_{\langle \sigma \rangle} \, \, | \, \, \sigma \in \sym(\Gamma) \}.$$ 
    
\begin{prop} \label{Orbisom} For all $f \in \aut(\AA)$, the orbit $O_f$ has a group structure and it is isomorphic to $\sym(\Gamma)$.
    \end{prop}
    \begin{proof}
    Let us define in $O_f$ the following binary operation:
    $$\begin{array}{cccclll}
\ast :& O_f \times O_f & \longrightarrow & O_f        &     &                 &  \vspace{0.2cm} \\
         & (f_{\langle \sigma \rangle },f_{\langle \gamma \rangle }) & \longmapsto          & f_{\langle \sigma \rangle }\ast f_{\langle \gamma \rangle }:=f_{ \langle \sigma \circ \gamma \rangle }.
\end{array} $$
\begin{enumerate}[noitemsep, leftmargin=20pt]
\item It is straightforward from definition of $\ast$ that $f=f_{\langle {\rm id}_{\Gamma} \rangle}$ is the identity element.
\item For $f_{\langle \sigma \rangle }\in O_f$, the element $f_{\langle \sigma^{-1} \rangle }$ is its inverse.
\item The associativity of $\ast$ is guaranteed by the associativity of the composition in $\sym(\Gamma).$ 
\end{enumerate}
Hence, $(O_f,\ast)$ is a group. Moreover, the map:
$$\begin{array}{cccclll}
\phi :& O_f  & \longrightarrow &    \sym(\Gamma)     &     &                 &  \vspace{0.2cm} \\
         & f_{\langle \sigma \rangle } & \longmapsto          & \sigma
\end{array} $$ is an isomorphism. From the definition, it is obviously a surjection and since the action is free, the map is also an injection. Moreover $$\phi(f_{\langle \sigma \rangle }\ast f_{\langle \gamma \rangle })=\phi(f_{ \langle \sigma \circ \gamma \rangle })=\sigma \circ \gamma=\phi(f_{\langle \sigma \rangle } )\circ \phi(f_{\langle \gamma \rangle }). \vspace{-0.6cm}$$
    \end{proof}
From the previous results, the set $\aut(\AA)$ can be described by the disjoint union of copies (as many as there are orbits) of $\sym(\Gamma) $. Hence,  a more precise description of $\aut(\AA)$ depends on a good set of representatives of the orbits. For example, if the action is transitive, we will have $\aut(\AA) \cong \sym(\Gamma)$. This is the case of perfect evolution $\mathbb{R}-$algebras.

 We will see that there is also a good description whenever for a given   $f \in \aut(\AA)$ there is $ \sigma \in \sym(\Gamma)$, depending on $f$, such that: 
\begin{equation}\label{inducidos}
f(e_i)=\mu_i e_{\sigma(i)},
\end{equation} 
with  $\mu_i \in \mathbb{K}^*, \textrm{ for }i=1,\dots,n.$ 
We will say that $f \in \aut(\AA)$ is {\it induced by a symmetry of the graph} if it satisfies
 \eqref{inducidos}.
This is the case of automorphisms of perfect evolution algebras (see \cite{Eld-Lab2015}), but it is not exclusive to them, as shown in the next example.

 \begin{exa}\label{ex:indper}
   For the graph  $P_3$  (see Fig. \ref{FIG:P3(a)}), 
     we have $\sym(P_3)=\{\rho_0, \rho_1\} \leq S_3$ with $\rho_0=(1)(2)(3)$ and $ \rho_1=(13)(2)$. Also, the evolution $\mathbb{C}$-algebra $\AA_{P_3}$, with natural basis $B=\{e_i\}_{i=1}^3$, has singular structure matrix given by:  \begin{equation*}
M_B= \begin{pmatrix}
    0 & 1 & 0 \\
    1 & 0 & 1 \\
    0 & 1 & 0
\end{pmatrix},
\end{equation*}
then $\AA_{P_3}$ is not perfect.
Moreover, we will show that  every $f \in\aut(\AA_{P_3})$ is induced by a symmetry of $P_3$. 
    Let $[f]_{B} =(t_{ik})_{i,k=1}^3$,  we want to know the possibilities for $t_{ik}$. By
    \eqref{colg2} in Corollary \ref{GemyColg},  
 \begin{equation} \label{eq:p3max1}
     t_{i 2}\neq 0, \hspace{0.2cm} \textrm{for exactly one} \hspace{0.2cm} i\in V. \vspace{0.2cm} 
 \end{equation}  
 
    Since $\mathcal{N}(1)=\mathcal{N}(3)$, then in \eqref{gemelos2} of Corollary \ref{GemyColg},  if $i=3$, then $t_{12}^2=t_{32}^2$ which together with \eqref{eq:p3max1} give: 
 \begin{equation} \label{eq:p3t2} t_{12}=t_{32}=0 \hspace{0.3cm} \textrm{and} \hspace{0.3cm} t_{22} \neq 0.
 \end{equation}
These, with $i=1$ and  $r=1,3$  in \eqref{eq:propb1} imply
 \begin{equation} \label{eq:p3t3}
  t_{21}=t_{23}=0.
  \end{equation}
  
 On the other hand, if $i=2$ and $r=1,3$ in \eqref{eq:propb1}, then 
\begin{equation} \label{eq:t22}
        t_{22}^2=t_{11}+t_{31} \hspace{0.5cm} \textrm{and} \hspace{0.5cm}  t_{22}^2=t_{13}+t_{33}, 
    \end{equation}
 and if $i=1,3$ and $r=2$ in \eqref{eq:propb1}, then 
\begin{equation} \label{eq:t11}
        t_{11}^2+ t_{13}^2=t_{22} \hspace{0.5cm} \textrm{and} \hspace{0.5cm}  t_{31}^2+t_{33}^2=t_{22}.
    \end{equation}
    In addition if $r=2$ in \eqref{eq:propa1}, 
    \begin{equation} \label{eq:t13}
        t_{i1}t_{j1}+t_{i3}t_{j3}=0 \hspace{0.5cm} \textrm{for all} \hspace{0.5cm} i\neq j. 
    \end{equation}
    Replacing \eqref{eq:t22} in \eqref{eq:t11} we obtain that
\begin{equation*}\label{eq:P31}
   t_{22}= (t_{22}^2-t_{31})^2+ t_{13}^2=t_{22}^4-2t_{22}^2t_{31}+t_{31}^2+t_{13}^2
\end{equation*}
and 
\begin{equation*}\label{eq:P32}
   t_{22}= (t_{22}^2-t_{13})^2+ t_{31}^2=t_{22}^4-2t_{22}^2t_{13}+t_{13}^2+t_{31}^2.
\end{equation*}
Then, by  \eqref{eq:p3t2}: $t_{13}=t_{31}$  and by \eqref{eq:t22}: $t_{11}=t_{33}$, and we have:
\begin{equation*}
[f]_B= \begin{pmatrix}
    t_{11} & 0 & t_{13} \\
   0  & t_{22}  & 0  \\
    t_{13}  & 0  & t_{11} 
\end{pmatrix}.
\end{equation*}
Notice that, if $i=1$ and $j=3$ in \eqref{eq:t13} then $0=t_{11}t_{31}+t_{13}t_{33}=2t_{11}t_{13}$. Since $f$ is a bijection, either $t_{11} \neq 0$ or $t_{13} \neq 0$. Hence,  all $f \in \aut(\AA_{P_3})$ is induced by a symmetry of $P_3$.
\end{exa}
 As stated in the introduction,
an evolution algebra with natural basis $\{e_i\}_{i=1}^n$ is {\it  2LI} if $\{e_i^2,e_j^2\}$ is linearly independent for $i \neq j$. In   \cite{BYM, Eld-Lab2021, CH}, it was shown that if a non-degenerate ($e_i^2\neq 0$ for all $i$)   evolution algebra  is {\it 2LI} then all its automorphisms are induced by symmetries of the associated graph. Since $\AA_{P_3}$ is not {\it 2LI}  then 
Example \ref{ex:indper} shows  that,
 in $\{\AA_{\Gamma}, \AA_{RW(\Gamma)}\}_{\Gamma \in \mathscr{G}}$, 
the family of  algebras whose automorphisms are all induced by symmetries of the associated graph, strictly contains the  family   of non-degenerate evolution algebras with the {\it 2LI} property. In particular, it contains strictly the family of  those that are perfect.
\begin{theorem}\label{orbitas} 
Let $\DD \leq \aut(\AA)$ be the diagonal subgroup, for $\AA \in \{\AA_{\Gamma},\AA_{RW(\Gamma)}\}$. If every element of $\aut(\AA)$  is induced by a symmetry of $\Gamma$ then $\aut(\AA)= \bigsqcup_{f\in \DD} O_f.$ Moreover,  $\aut(\AA)$ is finite and
    $|\aut(\AA)|=|\DD||\sym(\Gamma)|$.
    \end{theorem}
    \begin{proof}
       Let $g \in \aut(\AA)$ and $\sigma \in \sym(\Gamma)$ such that $g(e_i)=u_ie_{\sigma(i)}.$ Notice that  $h:=g_{\langle \sigma^{-1}\rangle} \in \DD$ and $O_h=O_g$. Thus, for every $g \in \aut(\AA)$,  there is $h \in \DD$ such that $g \in O_h$. Then $ \aut(\AA) \subseteq \bigcup_{f \in \DD}O_f$ and from the action  $ \bigcup_{f \in \DD}O_f \subseteq  \aut(\AA) $, hence $\aut(\AA)= \bigcup_{f\in \DD} O_f$. To show that it is a disjoint union, suppose that there exist $f,g \in \DD$ with $f \neq g$, we claim that $O_{f} \cap O_{g}= \emptyset.$ If not, there exist $\sigma_1, \sigma_2 \in \sym(\Gamma)$ such that $h={f}_{\langle \sigma_1 \rangle}=g_{\langle \sigma_2 \rangle}$ then ${f}_{\langle \sigma_1 \rangle}(e_i)= a_i e_{\sigma_1(i)}= b_i e_{\sigma_2(i)}=g_{\langle \sigma_2 \rangle}(e_i)$ for all $i$, which implies  $f=g$ and this is a contradiction.  
   
   On the other hand, since  $\aut(\AA)= \bigsqcup_{f\in \DD} O_f$ is finite and, from Proposition \ref{Orbisom}, for each $f$ we have  $O_f \cong \sym(\Gamma)$, then $|O_f|=|\sym(\Gamma)|$ and  $|\aut(\AA)|=|\DD||\sym(\Gamma)|$. \end{proof}
   \begin{exa} \label{ex:autp3}
 In Example \ref{ex:indper}, we show that every automorphism of $\AA(P_3)$ is induced either  by $\rho_0=(1)(2)(3)$ or by $ \rho_1=(13)(2)$, the symmetries of $P_3$. Hence, we can use Theorem \ref{orbitas} to exhibit ${\rm Aut}(\AA(P_3))$. 
 Since $P_3$ is  bipartite, Theorem \ref{D1o3} implies that $\DD \neq \{Id\}$. Then, by Corollary \ref{AlgD}, if $V_{1}=\{1,3\}$, $V_{2}=\{2\}$ and \,$1, \alpha, \beta$  are the three cubic roots of the unit in $\mathbb{C}$, then 
    $\DD= \{Id, h,g \}$
    with $ (1,1,1)_{[Id]_B}, \, (\alpha, \beta,\alpha)_{[h]_B}$ and $ (\beta,\alpha,\beta)_{[g]_B}$.     
   Hence 
   $O_{Id}=\{Id_{\langle \rho_0 \rangle}, Id_{\langle \rho_1 \rangle }\}, \, O_{h}=\{h_{\langle \rho_0 \rangle}, h_{\langle \rho_1 \rangle }\}, \, \, O_{g}=\{g_{\langle \rho_0 \rangle}, g_{\langle \rho_1 \rangle }\} , $
   then from Theorem \ref{orbitas}, 
   $$\aut(\AA_{P_3})=O_{Id} \cup O_{h} \cup O_{g}= $$ 
 \begin{eqnarray*}
  \left\{ \begin{pmatrix}   
 1 & 0 & 0  \\ 0 &  1 & 0  \\ 0 & 0 & 1    %
   \end{pmatrix} , \, \begin{pmatrix}   
  0 & 0 & 1 \\  0 & 1 & 0 \\  1 & 0 & 0   %
   \end{pmatrix}, \, \begin{pmatrix}   
 \alpha & 0 & 0  \\ 0 &  \beta & 0  \\ 0 & 0 & \alpha  %
   \end{pmatrix},\, \begin{pmatrix}   
  0 & 0 & \alpha \\   0 & \beta & 0 \\  \alpha & 0 & 0  %
   \end{pmatrix},   \,   \begin{pmatrix}   
 \beta & 0 & 0  \\ 0 &  \alpha & 0  \\ 0 & 0 & \beta   %
   \end{pmatrix}, \begin{pmatrix}   
  0 & 0 & \beta \\   0 & \alpha & 0 \\  \beta & 0 & 0  %
   \end{pmatrix} \,  \right\}. 
   \end{eqnarray*} 
  \end{exa}\vspace{0.3cm}

 The following corollary is straightforward from Theorem \ref{D1o3}  and Theorem \ref{orbitas}. 
  \begin{corollary}\label{trans} Let $\AA \in \{\AA_{\Gamma},\AA_{RW(\Gamma)}\}$ an evolution $\mathbb{K}$-algebra associated to $\Gamma$, for $\mathbb{K}$ algebraically closed.
  If every element of $\aut(\AA)$   is induced by a symmetry of $\Gamma$ then the following are equivalent:
  \begin{enumerate}[noitemsep, leftmargin=25pt]
  \item $ \aut (\AA) \cong \sym(\Gamma) $ 
  \item  $\DD=\{Id\}$ 
  \item $\Gamma$ is non-bipartite.
  \end{enumerate}
   \end{corollary}
  For perfect evolution algebras, this result  is not surprising considering the works in \cite{CLTV, CMTV,Eld-Lab2015, Eld-Lab2021, Sri-Zou2022}. However, the following example shows the existence of non-perfect evolution algebras satisfying  Corollary \ref{trans}.  
  \begin{exa} 
The Bull graph $B_{5}$ (see Fig. \ref{FIG:B5}) is  non-bipartite, by Lemma \ref{oddcycle}. Let $\AA_{B_{5}}$ be the associated evolution $\mathbb{C}$-algebra with natural basis $B=\{e_i\}_{i=1}^5$ and structure matrix $M_B$. 
\begin{figure}[htp]
\hspace{-6.5cm}
\begin{tikzpicture}[scale=0.35]
\filldraw [black] (4,0) circle (3.5pt);
\filldraw [black] (0,0) circle (3.5pt);
\filldraw [black] (2.1,-2) circle (3.5pt);
\filldraw [black] (-0.75,1.5) circle (3.5pt);
\filldraw [black] (4.7,1.5) circle (3.5pt);
\draw (-0.75,1.5) node[above,font=\footnotesize] {$1$} -- (0,0); 
\draw (4.7,1.5) node[above,font=\footnotesize] {$2$} -- (4,0);
\draw (-0.9,-1.4) node[above,font=\footnotesize] {$3$};
\draw (0,0) node[above,font=\footnotesize] {$\,$} - -  (4,0);
\draw (4.6,-1.4) node[above,font=\footnotesize] {$4$};
\draw (4,0) node[below,font=\footnotesize] {$\,$} -- (2.1,-2);
\draw (0,0) node[below,font=\footnotesize] {$\,$} -- (2.1,-2);
\draw (2.1,-3.4) node[above,font=\footnotesize] {$5$};
\hspace{4.2cm}
$ M_B= \begin{pmatrix}
    0 & 0 & 1 & 0 & 0 \\
    0 & 0 & 0 & 1 & 0 \\
    1 & 0 & 0 & 1 & 1 \\
    0 & 1 & 1 & 0 & 1 \\
    0 & 0 & 1 & 1 & 0 
\end{pmatrix}
$
\end{tikzpicture} 
\caption{ Bull graph $B_{5}$ and its structure matrix $M_B$.}\label{FIG:B5}
\end{figure} 
Since $M_B$ is singular then $\AA_{B_{5}}$
   is not perfect. 
Moreover, since  $\AA_{B_{5}}$ is
 non-degenerate and  $\{e_i^2,e_j^2\}$ is linearly independent for $i\neq j$  then $\AA_{B_{5}}$ is {\it 2LI}. Hence, from \cite[Corollary 2.7] {BYM},  $\AA_{B_{5}}$
 has  unique basis which implies, by \cite{BYM} and \cite{CH}, that its automorphisms are all induced by permutations in $\sym(B_5)$. Then, from Corollary \ref{trans}, $\aut(\AA_{B_{5}}) \cong \sym(B_5) \cong \mathbb{Z}_2.$
\end{exa}

\section{Relations between the automorphisms groups of $\AA_{\Gamma}$ and $\AA_{RW(\Gamma)}$ }\label{SRW}  
This section is devoted to pointing out some immediate consequences of the results in \S \ref{Res} about relations between the automorphism groups of $\AA_{\Gamma}$ and $\AA_{RW(\Gamma)}$.
Although these algebras are associated to the same graph  $\Gamma$, they are not necessarily isomorphic. 
Indeed, 
in \cite{PMPT} 2021, for a finite and non-singular graph $\Gamma$, it was shown that $\mathcal{A} _{\Gamma}\cong \mathcal{A}_{RW(\Gamma)}$ if and only if $\Gamma$ is a regular or a birregular graph.
In such cases, it is clear that $\aut(\AA_{RW(\Gamma)}) \cong \aut(\AA_{\Gamma})$. 
But, from the results in \S \ref{Res}, we will show that  it is not a necessary condition. 

\begin{exa}
The cycle graph $C_6$ (see Fig. \ref{FIG:C6})  is regular and non-singular, then by \cite[Theorem 2.3]{PMPT}, $\AA_{C_6} \cong \AA_{RW(C_6)}$ and thus $\aut(\AA_{RW(C_6)}) \cong \aut(\AA_{C_6})$. 
\begin{figure}[htp]
\begin{center}
\begin{tikzpicture}[scale=0.4]
\filldraw [black] (2.6,2) circle (3.5pt);
\filldraw [black] (1.3,0) circle (3.5pt);
\filldraw [black] (6.2,0) circle (3.5pt);
\filldraw [black] (2.6,-2) circle (3.5pt);
\filldraw [black] (4.9,2) circle (3.5pt);
\filldraw [black] (4.9,-2) circle (3.5pt);

\draw (4.9,2) node[above,font=\footnotesize] {$1$} -- (2.6,2); 
\draw (4.9,2) node[above,font=\footnotesize] {$\,$} -- (6.2,0); 
\draw (2.6,2) node[above,font=\footnotesize] {$2$} -- (1.3,0);
\draw (0.8,-0.5) node[above,font=\footnotesize] {$3$};
\draw (2.6,-2) node[below,font=\footnotesize] {$4$} -- (1.3,0);
\draw (4.9,-2) node[below,font=\footnotesize] {$5$} -- (2.6,-2);
\draw (6.2,0) node[above,font=\footnotesize] {$\,$} -- (4.9,-2);
\draw (6.6,-0.5) node[above,font=\footnotesize] {$6$};
\hspace{-1.2cm}
\end{tikzpicture}
\caption{ Cycle graph $C_{6}$.}\label{FIG:C6}
\end{center} 
\end{figure} 
Since the evolution $\mathbb{C}$-algebras $\AA_{C_6}$ and $\AA_{RW(C_6)}$ are perfect,  their automorphisms are all induced by symmetries of $C_6$, and from Theorem \ref{orbitas} we have that $\aut(\AA_{C_6})= \bigsqcup_{f\in \DD} O_f$.
Moreover, since $C_{6}$ is bipartite then, by Theorem \ref{D1o3} and Corollary \ref{AlgD},
$|\DD|=3$ 
and $|\aut(\AA_{RW(C_6)})|= |\aut(\AA_{C_6})|=|\DD| |\sym(C_6)|=36$, because  $|\sym(C_6)|=12$.
\end{exa}

Notice, however, that Theorem \ref{D1o3} and Corollary \ref{AlgD} imply the following immediate consequence that does not ask for a specific property of the graph to relate these automorphism groups.

\begin{corollary}\label{DandRWD}
  Let $\Gamma$ be a finite graph  and $\mathcal{D}_1 \leq \aut (\AA_{\Gamma})$ and $\mathcal{D}_2 \leq \aut (\AA_{RW(\Gamma)})$ the diagonal subgroups. Then  $\DD_1 \cong \DD_2$. 
     \end{corollary}    

\begin{exa}  \label{X}
     Let $\AA_{RW(K_{1,4})}$ be the evolution $\mathbb{C}$-algebra  associated  to the symmetric random walk  on the singular star graph $K_{1,4}$ (see Fig. \ref{FIG:K14}).
\begin{figure}[!h]
\begin{center}
\begin{tikzpicture}[scale=0.35]
\draw (3.9,0) node[above,font=\footnotesize] {$1$};
\draw (1.8,1) node[above,font=\footnotesize] {$2$} -- (3.9,0);;
\draw (5.9,1) node[above,font=\footnotesize] {$3$} -- (3.9,0);
\draw (1.8,-1) node[above,font=\footnotesize] {$\,$} -- (3.9,0);
\draw (1.8,-2.4) node[above,font=\footnotesize] {$4$};
\draw (5.9,-1) node[above,font=\footnotesize] {$\,$} -- (3.9,0);
\draw (5.9,-2.4) node[above,font=\footnotesize] {$5$}; 
\filldraw [black] (3.9,0) circle (3.5pt);
\filldraw [black] (1.8,1) circle (3.5pt);
\filldraw [black] (5.9,1) circle (3.5pt);
\filldraw [black] (1.8,-1) circle (3.5pt);
\filldraw [black] (5.9,-1) circle (3.5pt);
\hspace{-1.2cm}
\end{tikzpicture} \vspace{-0.1cm}
\caption{ Star graph $K_{1,4}$}
\label{FIG:K14}
\end{center} 
\end{figure} 
Note that $\AA_{RW(K_{1,4})}$ is not perfect and does not satisfy property 2LI. That is, even if we do not have enough hypotheses  to relate the structure of $\aut(\AA_{K_{1,4}})$ and $\aut(\AA_{RW(K_{1,4})})$, from Corollary \ref{DandRWD} we know that $\DD_1 \cong \DD_2$. Moreover, since $K_{1,4}$ is  bipartite, by Theorem \ref{D1o3} and Corollary \ref{AlgD}, we have
  $\DD_1= \{Id, h,g \}$
    with $ (1,1,1,1,1)_{[Id]_B}, \, (\alpha, \beta,\beta,\beta,\beta)_{[h]_B}$ and $ (\beta,\alpha,\alpha,\alpha,\alpha)_{[g]_B}$. 
  \end{exa}

\begin{corollary}\label{RWcorigaut1}
If  the elements of  $\aut(\AA_{\Gamma})$ and the elements of $\aut(\AA_{RW(\Gamma)})$ are all induced by  symmetries of $\Gamma$, then  $\aut(\AA_{\Gamma}) \cong  \aut(\AA_{RW(\Gamma)}).$
\end{corollary}

We will end this section with an example of a singular graph $\Gamma$ in which $\AA_{\Gamma}$ and $\AA_{RW(\Gamma)}$ are not isomorphic but their respective automorphism groups are isomorphic. Moreover, the algebra $\AA_{RW(\Gamma)}$  is  not perfect and  not {\it 2LI}, but its automorphisms are all induced by symmetries of the graph, which allows us to calculate its group of automorphisms. This example performs similar counts to Example \ref{ex:indper}, except that we work not only with structure constants 0 and 1 given by the graph adjacency matrix, but with structure constants of the graph transition matrix, which can be between 0 and 1. 

\begin{exa} \label{GraphT41}
Let $T_{4,1}$  be the tadpole graph  (see Fig. \ref{FIG:T41}).
\begin{figure}[!h]
\begin{center}
\begin{tikzpicture}[scale=0.4]

\draw (3.9,0) node[above,font=\footnotesize] {$2$} -- (5.9,1.5); 
\draw (5.9,1.5) node[above,font=\footnotesize] {$1$} -- (7.9,0);
\draw (7.9,0) node[above,font=\footnotesize] {$4$} -- (5.9,-1.7);
\draw (5.9,-1.7) node[above,font=\footnotesize] {$\,$} -- (3.9,0);
\draw (5.9,-3) node[above,font=\footnotesize] {$3$};
\draw (10.9,0) node[above,font=\footnotesize] {$5$} -- (7.9,0);
\filldraw [black] (3.9,0) circle (3.5pt);
\filldraw [black] (5.9,1.5) circle (3.5pt);
\filldraw [black] (7.9,0) circle (3.5pt);
\filldraw [black] (5.9,-1.7) circle (3.5pt);
\filldraw [black] (10.9,0) circle (3.5pt);
\hspace{-1.2cm}
\end{tikzpicture}
\caption{ Tadpole graph $T_{4,1}$.}\label{FIG:T41}
\end{center} 
\end{figure}  
The evolution $\mathbb{C}$-algebra $\AA_{RW(T_{4,1})}$, with natural basis $B=\{e_i\}_{i=1}^5$, is neither perfect nor  2LI.
From the graph, we conclude that  $\sym(T_{4,1})=\{\rho_0, \rho_1\} \leq S_5$, with $\rho_0=(1)(2)(3)(4)(5)$ and $\rho_1=(13)(2)(4)(5)$.  

Let   $f \in \aut(\AA_{RW(T_{4,1})})$ and $[f]_{B} =(t_{ik})_{i,k=1}^n$ its representation matrix in the basis $B$. Since $f$ is an automorphism, the matrix has no null rows or  columns. We want to know the possibilities for $t_{ik} \neq 0$, and for this we will use the equations in Proposition  \ref{eqayba1a1} and  Corollary \ref{GemyColg} repeatedly. 
For $\ell=5$ in   \eqref{colg1}:   \begin{equation} \label{eq:tp1}
        t_{i4}t_{j4}=0 \hspace{0.2cm} \textrm{for all} \hspace{0.2cm} i\neq j,
    \end{equation}
hence, and  from $r=1$ in \eqref{RWeq:propa2}: 
    \begin{equation} \label{eq:tp2}
       t_{i2}t_{j2}=0 \hspace{0.5cm} \textrm{for all} \hspace{0.5cm} i\neq j. 
    \end{equation}
Also in \eqref{RWeq:propa2}, for $r=2,4:$
\begin{equation} \label{eq:tp13}
        \frac{1}{2}t_{i1}t_{j1}+\frac{1}{2}t_{i3}t_{j3}=0 \hspace{0.5cm} \textrm{and} \hspace{0.5cm} \frac{1}{2}t_{i1}t_{j1}+\frac{1}{2}t_{i3}t_{j3}+t_{i5}t_{j5}=0, 
    \end{equation}
then
\begin{equation}    \label{eq:tp4}
t_{i5}t_{j5}=0, 
         \hspace{0.5cm} \textrm{for all} \hspace{0.5cm} i\neq j.
         \end{equation}
It follows by \eqref{eq:tp1}, \eqref{eq:tp2} and \eqref{eq:tp4} that,
    \begin{equation} \label{eq:tp5}
     t_{i k}\neq 0, \hspace{0.2cm} \textrm{for exactly one} \hspace{0.2cm} i\in V, \hspace{0.2cm} \textrm{when} \hspace{0.2cm} k=2,4,5. \end{equation}
    
    On the other hand,  since $\mathcal{N}(1)=\mathcal{N}(3)$,   for $i=5$ and  $i=3$ in \eqref{gemelos1} we obtain: 
 \begin{equation} \label{eq:tp6}t_{41}=t_{43} \hspace{0.3cm} \textrm{and} \hspace{0.3cm}  t_{21}=t_{23},\end{equation}
 and 
 if $i=2$ and  $i=4:$
 \begin{equation} \label{eq:tp7}
 t_{11}+t_{31}=t_{13}+t_{33} \hspace{0.3cm} \textrm{and} \hspace{0.3cm}  t_{51}=t_{53}.\end{equation}
Now,  in \eqref{gemelos2},  if $i=5$ and  $i=1:$
 \begin{equation} \label{eq:tp8}t_{14}^2=t_{34}^2 \hspace{0.3cm} \textrm{and} \hspace{0.3cm}  t_{12}^2=t_{32}^2,\end{equation}
 and 
 if $i=2$ and  $i=4:$
 \begin{equation} \label{eq:tp9}
 t_{11}^2+t_{13}^2=t_{31}^2+t_{33}^3 \hspace{0.3cm} \textrm{and} \hspace{0.3cm}  t_{15}^2=t_{35}^2.\end{equation}
Hence, from \eqref{eq:tp5},  \eqref{eq:tp8} and \eqref{eq:tp9} we obtain:
\begin{equation} \label{tpceros1} t_{12}=t_{32}=t_{14}=t_{34}=t_{15}=t_{35}=0. 
\end{equation}
Then, in \eqref{RWeq:propb2},  for $i=2$ and $r=5: \;$
$\frac{1}{3} t_{24}^2=\frac{1}{2}(t_{15}+t_{35})=0, \textrm{ and }  t_{24}=0.$
Analogously, for $i=1$ and $r=5: \;$
$
\frac{1}{2}(t_{25}+t_{45})=\frac{1}{3} t_{14}^2=0, \textrm{ then from \eqref{eq:tp5}: } t_{25}=t_{45}=0.
$ Hence, for $i=r=5: \;$ $ \frac{1}{3}t_{54}^2=t_{45}=0.
$ Compiling these information, we have:
\begin{equation} \label{tpceros2} t_{25}=t_{45}=t_{24}=t_{54}=0 .\end{equation}
This in \eqref{RWeq:propb2}, for $i=r=4$: 
$
t_{41}^2=0 \textrm{ then, from }
 \eqref{eq:tp6}:  t_{41}=t_{43}=0.
$

 For $i=5$ and $r=1:$  $t_{52}^2=0$,  for $i=r=1: t_{21}=0$, for  $i=r=3: t_{23}=0$,  so
\begin{equation} \label{tpceros3} t_{21}=t_{41}=t_{23}=t_{43}=t_{52}=0.\end{equation}
Since $f$ is an automorphism, from \eqref{tpceros1}, \eqref{tpceros2} and \eqref{tpceros3} we conclude that:
\begin{equation}\label{no-cero}
 t_{ii}\neq 0 \textrm{ for } i=2,4,5,
 \end{equation}
  thus, from \eqref{eq:tp5}:  $t_{42}=t_{52}=0$.
  This, together with \eqref{eq:tp7},  in \eqref{RWeq:propb2} for $i=5$ and $r=2:$ $t_{51}=t_{53}=0$.
 Compiling these information, we have: 
\begin{equation} \label{tpceros4} t_{42}=t_{52}=t_{51}=t_{53}=0.\end{equation}
From \eqref{tpceros1},\eqref{tpceros2},\eqref{tpceros3}, \eqref{no-cero} and \eqref{tpceros4},  we notice that $t_{11},t_{13},t_{31},t_{33}$  we do not know yet and that the other entries are all equal to zero except $t_{22},t_{44}$ and $t_{55}$.
This information,  in \eqref{RWeq:propb2} for $i=1,r=2$ and $i=3,r=2$, gives:
\begin{equation*} \label{eq:tp11}
        t_{11}^2+ t_{13}^2=t_{22} \hspace{0.5cm} \textrm{and} \hspace{0.5cm}  t_{31}^2+t_{33}^2=t_{22}, 
    \end{equation*}
 and,   for $i=2,r=1$ and $i=2,r=3$, gives:  
\begin{equation*} \label{eq:tp22}
        t_{22}^2=t_{11}+t_{31} \hspace{0.5cm} \textrm{and} \hspace{0.5cm}  t_{22}^2=t_{13}+t_{33}. 
    \end{equation*}
The last four equations are equal to  the equations in \eqref{eq:t22}, \eqref{eq:t11}  in the Example \ref{ex:indper}, as well as  the first  equation in \eqref{eq:tp13} with the equation  \eqref{eq:t13}, in the same example. 
Then, with the same arguments used in Example \ref{ex:indper}, we obtain that
 every $f \in \aut(\AA_{RW(T_{4,1})})$ is induced by a symmetry of $T_{4,1}$.
 
 Since $T_{4,1}$ is a bipartite graph, then 
$ \DD_2 \neq \{Id\}$ and   $\DD_2= \{Id, h,g \}$
   with $ (1,1,1,1,1)_{[Id]_B}, $ $\, (\alpha, \beta,\alpha,\beta,\alpha)_{[h]_B}$ and $ (\beta,\alpha,\beta,\alpha,\beta)_{[g]_B}$. 
   Moreoever, 
   $O_{Id}=\{Id_{\langle \rho_0 \rangle}, Id_{\langle \rho_1 \rangle }\},$ $O_{h}=\{h_{\langle \rho_0 \rangle}, h_{\langle \rho_1 \rangle }\},$ and $O_{g}=\{g_{\langle \rho_0 \rangle}, g_{\langle \rho_1 \rangle }\} ,$ 
   then from Proposition \ref{orbitas}: 
   $$\aut(\AA_{RW(T_{4,1})})=O_{Id} \cup O_{h} \cup O_{g}= \vspace{-0.4cm}$$
       \begin{small}
   \begin{eqnarray*}
      \left\{\begin{pmatrix}   
 1 &\hspace{-5pt} 0 &\hspace{-5pt} 0 &\hspace{-5pt} 0 &\hspace{-5pt} 0   \\ 0 &\hspace{-5pt}  1 &\hspace{-5pt} 0 &\hspace{-5pt} 0 &\hspace{-5pt} 0    \\ 0 &\hspace{-5pt} 0 &\hspace{-5pt} 1 &\hspace{-5pt} 0 &\hspace{-5pt} 0  \\ 0 &\hspace{-5pt}  0 &\hspace{-5pt} 0 &\hspace{-5pt} 1 &\hspace{-5pt} 0    \\ 0 &\hspace{-5pt} 0 &\hspace{-5pt} 0 &\hspace{-5pt} 0 &\hspace{-5pt} 1     %
   \end{pmatrix} ,  \begin{pmatrix}   
 0 &\hspace{-5pt} 0 &\hspace{-5pt} 1 &\hspace{-5pt} 0 &\hspace{-5pt} 0   \\ 0 &\hspace{-5pt}  1 &\hspace{-5pt} 0 &\hspace{-5pt} 0 &\hspace{-5pt} 0    \\ 1 &\hspace{-5pt} 0 &\hspace{-5pt} 0 &\hspace{-5pt} 0 &\hspace{-5pt} 0  \\ 0 &\hspace{-5pt}  0 &\hspace{-5pt} 0 &\hspace{-5pt} 1 &\hspace{-5pt} 0    \\ 0 &\hspace{-5pt} 0 &\hspace{-5pt} 0 &\hspace{-5pt} 0 &\hspace{-5pt} 1     %
   \end{pmatrix} , \begin{pmatrix}   
 \alpha &\hspace{-5pt} 0 &\hspace{-5pt} 0 &\hspace{-5pt} 0 &\hspace{-5pt} 0   \\ 0 &\hspace{-5pt}  \beta &\hspace{-5pt} 0 &\hspace{-5pt} 0 &\hspace{-5pt} 0    \\ 0 &\hspace{-5pt} 0 &\hspace{-5pt} \alpha &\hspace{-5pt} 0 &\hspace{-5pt} 0  \\ 0 &\hspace{-5pt}  0 &\hspace{-5pt} 0 &\hspace{-5pt} \beta &\hspace{-5pt} 0    \\ 0 &\hspace{-5pt} 0 &\hspace{-5pt} 0 &\hspace{-5pt} 0 &\hspace{-5pt} \alpha     %
   \end{pmatrix} ,  \begin{pmatrix}   
 0 &\hspace{-5pt} 0 &\hspace{-5pt} \alpha &\hspace{-5pt} 0 &\hspace{-5pt} 0   \\ 0 &\hspace{-5pt}  \beta &\hspace{-5pt} 0 &\hspace{-5pt} 0 &\hspace{-5pt} 0    \\ \alpha &\hspace{-5pt} 0 &\hspace{-5pt} 0 &\hspace{-5pt} 0 &\hspace{-5pt} 0  \\ 0 &\hspace{-5pt}  0 &\hspace{-5pt} 0 &\hspace{-5pt} \beta &\hspace{-5pt} 0    \\ 0 &\hspace{-5pt} 0 &\hspace{-5pt} 0 &\hspace{-5pt} 0 &\hspace{-5pt} \alpha     %
   \end{pmatrix},   \begin{pmatrix}   
 \beta &\hspace{-5pt} 0 &\hspace{-5pt} 0 &\hspace{-5pt} 0 &\hspace{-5pt} 0   \\ 0 &\hspace{-5pt}  \alpha &\hspace{-5pt} 0 &\hspace{-5pt} 0 &\hspace{-5pt} 0    \\ 0 &\hspace{-5pt} 0 &\hspace{-5pt} \beta &\hspace{-5pt} 0 &\hspace{-5pt} 0  \\ 0 &\hspace{-5pt}  0 &\hspace{-5pt} 0 &\hspace{-5pt} \alpha &\hspace{-5pt} 0    \\ 0 &\hspace{-5pt} 0 &\hspace{-5pt} 0 &\hspace{-5pt} 0 &\hspace{-5pt} \beta     %
   \end{pmatrix} ,    \begin{pmatrix}   
 0 &\hspace{-5pt} 0 &\hspace{-5pt} \beta &\hspace{-5pt} 0 &\hspace{-5pt} 0   \\ 0 &\hspace{-5pt}  \alpha &\hspace{-5pt} 0 &\hspace{-5pt} 0 &\hspace{-5pt} 0    \\ \beta &\hspace{-5pt} 0 &\hspace{-5pt} 0 &\hspace{-5pt} 0 &\hspace{-5pt} 0  \\ 0 &\hspace{-5pt}  0 &\hspace{-5pt} 0 &\hspace{-5pt} \alpha &\hspace{-5pt} 0    \\ 0 &\hspace{-5pt} 0 &\hspace{-5pt} 0 &\hspace{-5pt} 0 &\hspace{-5pt} \beta     %
   \end{pmatrix}  \right\}.
   \end{eqnarray*}
   \end{small}
   
 By Corollary \ref{DandRWD},  $\DD_1 \cong \DD_2$. 
   Since the same calculations made for $f \in \aut(\AA_{RW(T_{4,1})})$  apply for $f_1 \in \aut(\AA_{T_{4,1}})$, we  conclude that all $f_1 \in \aut(\AA_{T_{4,1}})$ is induced by a symmetry of $T_{4,1}$. Then, by Corollary \ref{RWcorigaut1}, we have
   $\aut(\AA_{RW(T_{4,1})}) \cong \aut(\AA_{T_{4,1}})$.
 \end{exa}

  \begin{remark}
  Notice that, if in Example \ref{GraphT41}, we consider the algebras on the field $\mathbb{R}$ then $\DD= \{Id\}$ and $O_{Id}=\{Id_{\langle \rho_0 \rangle}, Id_{\langle \rho_1 \rangle }\}$.  
Thus, with the same calculation in Example \ref{GraphT41}, we show   that   $\aut(\AA_{RW(T_{4,1})})=O_{Id}$ and  $\aut(\AA_{RW(T_{4,1})}) \cong \aut(\AA_{T_{4,1}})$, although in \cite[Proposition 3.3]{PMPS} it has been shown that $\AA_{T_{4,1}} \ncong \AA_{RW(T_{4,1})}.$
    \end{remark}
\section*{Acknowledgments}
Part of this research was carried out during a visit by  M. L. Rodi\~no Montoya to the Universidade Federal de S\~ao Carlos (UFSCar) and during a visit by  N. Viana Bedoya to the Universidad de Antioquia (UdeA). The authors extend their gratitude to this institutions for their hospitality and support. This work was supported by the  Funda\c{c}\~ao de Amparo \`a Pesquisa do Estado de S\~ao Paulo - FAPESP (Grant 2022/08948-2).

\end{document}